\documentclass[11pt, a4paper]{article}

\usepackage[utf8]{inputenc}
\usepackage[T1]{fontenc}
\usepackage{amsmath, amssymb, amsthm}
\usepackage{geometry}
\usepackage{caption}
\usepackage{tikz} 
\usepackage[hidelinks]{hyperref}
\newtheorem{theorem}{Theorem}

\newtheorem{proposition}{Proposition}

\begin{document}
\title{An Algebraic Approach to the Arithmetic of Square Tilings}
\author{
    Paul Perrier\thanks{Email: \href{mailto:paul.perrier@polytechnique.org}{\textbf{paul.perrier@polytechnique.org}}} \\
    \small École Polytechnique, 91128 Palaiseau Cedex, France
}

\date{Mars 2026} 

\maketitle

\begin{abstract}
It is a classical result dating back to Dehn (1903) that squares tiling a rectangle must have rational side lengths. However, bounding the arithmetic complexity of these tilings, specifically the growth of their denominators, is an active area of research. Recent contributions have tackled these arithmetic constraints through the complexity of planar multigraphs and Diophantine approximations. In this paper, we offer an algebraic perspective of this problem by encoding the tiling's geometry into a block matrix. This approach allows us to reduce geometric constraints to a finite set of additive relations. By exploiting the total unimodularity of interval matrices (Ghouila-Houri, 1962), we provide a self-contained proof that the least common multiple of the denominators is bounded by $2^n$ in two dimensions, where $n$ is the order of the tiling. Furthermore, because we can naturally generalizes this matrix to represent $d$-dimensional hypercube tilings, yielding a bound of $d^n$. We demonstrate the utility of this framework by providing a brief proof of Richard Kenyon's 1996 theorem regarding the minimum number of squares required to tile an integer rectangle.

\vspace{0.5cm}
\end{abstract}

\section*{Introduction}

The problem of tiling a rectangle with squares, commonly referred to as "squaring the rectangle", occupies a unique place in the history of combinatorial geometry. As early as 1903, Max Dehn \cite{Dehn1903} established a fundamental qualitative result: a rectangle can be tiled by squares only if its sides are commensurable. In 1940, Brooks, Smith, Stone, and Tutte \cite{BSST1940} revolutionized the study of these tilings by establishing an analogy with planar electrical networks, where currents correspond to side lengths and voltages to segment positions.

Finding the maximum size of the squares' common denominators is a problem that has seen several recent developments. For instance, Fomin (2021) \cite{Fomin2021} bounded this arithmetic scaling by leveraging multigraphs, achieving a tight upper bound for the two-dimensional case. More recently, Keleti et al. (2023) \cite{Liu2023} employed Diophantine approximations to study this problem, an approach that also allowed them to generalize denominator bounds to higher dimensional hypercube tilings.

This work proposes a different, linear-algebraic perspective. By constructing an "extended grid" from the tiling's segments, we show that the geometry of the tiling can be entirely encapsulated within an incidence matrix $M$. This approach transforms geometric constraints into algebraic properties of the null space of $M$, offering a direct view of the problem's arithmetic structure. While our resulting bounds do not asymptotically improve upon the most recent results, this matrix formulation offers a conceptually simple, self-contained framework. Whereas approaches relying on electrical frameworks yield results solely for two dimensional tilings, this approach can be easily generalized to higher dimensions. 

The paper is organized as follows. In Section 1, we introduce the extended grid construction and recall Dehn's rationality proof (Theorem \ref{thm:rationality}) through the modern formalism of $\mathbb{Q}$-linear maps \cite{AignerZiegler}. We establish a foundational result (Proposition \ref{prop:segments_bound}) bounding the number of grid segments by $n+1$, demonstrated via geometric mapping and later via algebraic rank arguments.

In Section 2, we formalize the construction of the incidence matrix $M$. We prove the central result (Theorem \ref{thm:kernel_dim}) that $\dim \ker M = 1$, ensuring that the matrix characterizes the tiling independently from scaling. This formulation also yields a combinatorial upper bound of $O(c^n n^{3n+2})$ on the number of tilings of order $n$ (Theorem \ref{thm:combinatorial_bound}).

Section 3 constitutes the arithmetic core of the paper. By exploiting the interval matrix structure of the incidence blocks \cite{GhouilaHouri1962}, we show that the subdeterminants of the matrix are linked to totally unimodular matrices. From this, we derive our main result (Theorem \ref{thm:den_bound}): the least common multiple of the denominators $D$ satisfies $D \leq 2^n$.

In Section 4, we demonstrate the efficacy of this approach by recovering Richard Kenyon's theorem \cite{Kenyon1996}, through a concise proof using our bound on $D$, proving that a rectangle of dimensions $p \times q$ requires at least $\max(q/p, \log_2 q)$ squares to tile it. 

Finally, in Section 5, we detail how this interval matrix structure naturally extends to $d$-dimensional hypercuboids, yielding a generalized denominator bound of $d^n$.

\section{The Extended Grid and Rationality Constraints}

Let us consider a rectangle tiled by a finite number of squares. By applying a suitable homothety, we can assume without loss of generality that the dimensions of the rectangle are $a \times 1$, where $a \in \mathbb{R}_+^*$. We denote by $x_1, \dots, x_n$ the side lengths of the $n$ squares composing the tiling. Throughout this paper, the \textit{order} of a tiling refers to the number $n \in \mathbb{N}^*$ of squares it contains. From now on, we fix the order $n$ of the tilings we consider.

To ensure that our future algebraic representation is unique and free from index permutation symmetries, we impose a canonical ordering on these $n$ squares. We label them $x_1, \dots, x_n$ by sorting their bottom-left corners lexicographically: first by their abscissa (from left to right), and then by their ordinate (from bottom to top). This canonical labeling is illustrated in Figure 1.

Before proving the main theorems, we introduce a geometric construction that will be fundamental throughout this paper. 

\vspace{0.5cm}
\noindent
\begin{minipage}[c]{0.55\textwidth}
    Given a tiling of an $a \times 1$ rectangle by $n$ squares, we extend all the line segments that form the boundaries of the squares to the edges of the rectangle. This partitions the rectangle into a finer grid of smaller rectangles, which we will call the \textbf{extended grid}.

    Let $y_1, \dots, y_k$ be the lengths of the horizontal segments appearing in this extended grid, ordered from left to right. Similarly, let $z_1, \dots, z_l$ be the lengths of the vertical segments, ordered from bottom to top. 
    
    Since each square has its right side extended at most once, and its left side is either the right side of another square or the boundary of the rectangle, there can be at most $n$ vertical lines. Therefore, $k \leq n$. By a symmetric argument, $l \leq n$. 
    
    For each square $x_i$ ($1 \leq i \leq n$), we define $Y_i$ as the set of all $y_j$ that intersect the square $x_i$, and $Z_i$ as the set of all $z_j$ that intersect it. For example, in Figure 1, we have $Y_3 = \{y_1, y_2\}$ and $Z_3 = \{z_3, z_4\}$.
\end{minipage}%
\hfill%
\begin{minipage}[c]{0.4\textwidth}
    \centering
    \resizebox{\textwidth}{!}{
    \begin{tikzpicture}[scale=1.2]
        \draw[dashed, gray, thick] (1,0) -- (1,5);
        \draw[dashed, gray, thick] (3,0) -- (3,5);
        \draw[dashed, gray, thick] (4,0) -- (4,5);
        
        \draw[dashed, gray, thick] (0,1) -- (5,1);
        \draw[dashed, gray, thick] (0,2) -- (5,2);
        \draw[dashed, gray, thick] (0,3) -- (5,3);

        \draw[ultra thick] (0,2) rectangle (3,5) node[midway] {$x_3$};
        \draw[ultra thick] (3,3) rectangle (5,5) node[midway] {$x_7$};
        \draw[ultra thick] (3,2) rectangle (4,3) node[midway] {$x_6$};
        \draw[ultra thick] (4,2) rectangle (5,3) node[midway] {$x_8$};
        \draw[ultra thick] (3,0) rectangle (5,2) node[midway] {$x_5$};
        \draw[ultra thick] (0,1) rectangle (1,2) node[midway] {$x_2$};
        \draw[ultra thick] (0,0) rectangle (1,1) node[midway] {$x_1$};
        \draw[ultra thick] (1,0) rectangle (3,2) node[midway] {$x_4$};

        \draw[ultra thick] (0,0) rectangle (5,5);

        \foreach \x in {0,1,3,4,5} {
            \draw[dotted, thick] (\x,0) -- (\x,-0.8);
        }
        \draw[<->, thick] (0.02,-0.6) -- (0.98,-0.6) node[midway, below] {$y_1$};
        \draw[<->, thick] (1.02,-0.6) -- (2.98,-0.6) node[midway, below] {$y_2$};
        \draw[<->, thick] (3.02,-0.6) -- (3.98,-0.6) node[midway, below] {$y_3$};
        \draw[<->, thick] (4.02,-0.6) -- (4.98,-0.6) node[midway, below] {$y_4$};

        \foreach \y in {0,1,2,3,5} {
            \draw[dotted, thick] (0,\y) -- (-0.8,\y);
        }
        \draw[<->, thick] (-0.6,0.02) -- (-0.6,0.98) node[midway, left] {$z_1$};
        \draw[<->, thick] (-0.6,1.02) -- (-0.6,1.98) node[midway, left] {$z_2$};
        \draw[<->, thick] (-0.6,2.02) -- (-0.6,2.98) node[midway, left] {$z_3$};
        \draw[<->, thick] (-0.6,3.02) -- (-0.6,4.98) node[midway, left] {$z_4$};
    \end{tikzpicture}
    }
    
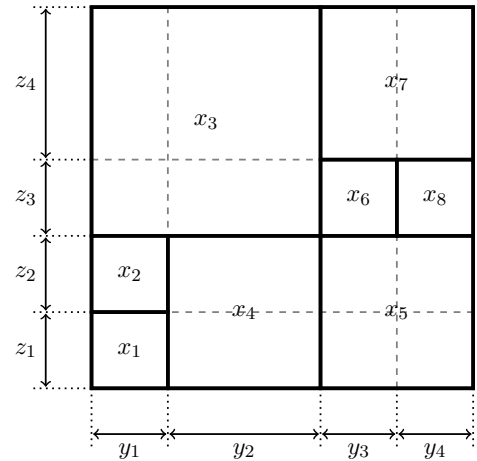
\captionof{figure}{Extended grid for a tiling of order $n=8$ of a square ($a=1$).}
\end{minipage}
\vspace{0.5cm}

We now recall the following foundational theorem, initially demonstrated by Max Dehn in 1903 \cite{Dehn1903}. The elegant algebraic proof presented here is adapted from Aigner and Ziegler's \textit{Proofs from THE BOOK} \cite{AignerZiegler}, slightly expanded to explicitly demonstrate the rationality of the constituent squares' dimensions alongside the rectangle's proportions.

\begin{theorem} \label{thm:rationality}
    For any square tiling of an $a \times 1$ rectangle, we have $a \in \mathbb{Q}$ and $x_i \in \mathbb{Q}$ for all $1 \leq i \leq n$.
\end{theorem}

\begin{proof}
    Suppose, for the sake of contradiction, that there exists an index $i$ ($1 \leq i \leq n$) such that $x_i \notin \mathbb{Q}$. Let $E = \operatorname{Span}_{\mathbb{Q}}(y_1, \dots, y_k, z_1, \dots, z_l)$ be the $\mathbb{Q}$-vector space spanned by the segment lengths of the extended grid. 
    
    Notice that $1 \in E$ because $\sum\limits_{j=1}^l z_j = 1$. Furthermore, $x_j \in E$ because $\sum\limits_{y \in Y_j} y = \sum\limits_{z \in Z_j} z = x_j$ for all $1\leq j\leq n$. Since $1 \in E$, $x_i \in E$, and $x_i$ is irrational, it follows that $\dim_{\mathbb{Q}}(E) \geq 2$.
    
    There thus exists a $\mathbb{Q}$-linear map $\phi : E \to \mathbb{R}$ such that $\phi(1) = 0$ and $\phi(x_i) = 1$. Let us introduce a bilinear map corresponding to a pseudo-area for rectangles, defined by $A(r_1, r_2) = \phi(r_1)\phi(r_2)$. For any square of size $x_j$, its pseudo-area is $A(x_j) = \phi(x_j)^2$. The total pseudo-area of the rectangle is conserved:
    \begin{align*}
        A(a, 1) = \phi(a)\phi(1) &= \left( \sum_{u=1}^k \phi(y_u) \right) \left( \sum_{v=1}^l \phi(z_v) \right) \\
        &= \sum_{u=1}^k \sum_{v=1}^l \phi(y_u)\phi(z_v) \\
        &= \sum_{j=1}^n \sum_{y \in Y_j} \sum_{z \in Z_j} \phi(y)\phi(z) \\
        &= \sum_{j=1}^n \phi(x_j)^2 = \sum_{j=1}^n A(x_j).
    \end{align*}
    However, since $\phi(1) = 0$, we have $A(a, 1) = 0$. On the other hand, $A(x_j) = \phi(x_j)^2 \geq 0$ for all $1 \leq j \leq n$, and specifically $A(x_i) = \phi(x_i)^2 = 1$. This yields $\sum\limits_{j=1}^n A(x_j) \geq 1 > 0$, which contradicts $A(a, 1) = 0$. 
    
    Consequently, all square sizes $x_i$ must be rational. It immediately follows that $a = \sum\limits_{j=1}^k y_j \in \mathbb{Q}$, concluding the proof.
\end{proof} 

Having established the rationality of the tiling, we will find a bound on the number of line segments in the extended grid in relation to the order $n$ of the tiling. This result will be very useful to refine our combinatorial and algebraic bounds later on. 

To highlight the duality of our approach, we will demonstrate this proposition in two completely different ways: first, immediately below, using a direct geometric mapping argument; and second, in Section 3, using an algebraic argument derived from the rank of our incidence matrix.

\begin{proposition} \label{prop:segments_bound}
    With the notations defined above, we have $k + l \leq n + 1$.
\end{proposition}

\begin{proof}[Geometric Proof]
    Let $\mathcal{V}$ be the set of the $k$ vertical lines that form the left boundaries of the $k$ vertical strips (columns) in the grid. Note that the rightmost boundary of the rectangle is not included in $\mathcal{V}$.
    Similarly, let $\mathcal{H}$ be the set of the $l$ horizontal lines that form the bottom boundaries of the $l$ horizontal strips (rows) in the grid. The topmost boundary of the rectangle is not included in $\mathcal{H}$. 
    The total number of these lines is $|\mathcal{V}| + |\mathcal{H}| = k + l$.

    We construct a mapping from $\mathcal{V} \cup \mathcal{H}$ to the set of the $n$ squares in the tiling by associating each line to the square located at its "origin":
    \begin{itemize}
        \item For each vertical line $V \in \mathcal{V}$, we consider its lowest point that is the bottom left corner of a square and we associate $V$ with this specific square.
        \item For each horizontal line $H \in \mathcal{H}$, we consider its leftmost point that is the bottom left corner of a square and we associate $H$ with this square.
    \end{itemize}

    Let us determine how many lines can be associated with any single square $S$. Suppose a square $S$, whose bottom-left corner is denoted by $P$, is associated with both a vertical line $V \in \mathcal{V}$ and a horizontal line $H \in \mathcal{H}$. 
    
    This implies that $P$ is the lowest point of $V$ and the leftmost point of $H$. Consequently, no segment of the grid extends below $P$, and no segment extends to the left of $P$. Since $P$ is a corner of a square belonging to the tiling, this geometric configuration is only possible if $P$ is the global bottom-left corner of the entire $a \times 1$ rectangle.

    Therefore, the square located at the bottom-left of the rectangle is associated with exactly $2$ lines (the left border and the bottom border of the rectangle). Every other square in the tiling (there are $n-1$ of them) serves as the origin for at most $1$ line.
    
    By summing these associations, we obtain an upper bound on the total number of lines:
    \[
    k + l \leq 2 + (n - 1) = n + 1.
    \]
\end{proof}

\section{Incidence Matrix}

The core idea of this paper is to generalise the argument of Theorem \ref{thm:rationality}. Notice that the function $\phi$ does not need to be a linear map over a vector space, it merely needs to satisfy a finite set of additive relations derived from the geometry of the tiling. Moreover, it is not necessary to define such a function, but only to associate a variable to each of the numbers $1, a, x_1, \dots, x_n, y_1, \dots, y_k, z_1, \dots, z_l$ that verifies some relations.

The additive relations are formalised using a matrix $M$ called the incidence matrix of a tiling.

Let us define the column vector of the formal geometric variables: 
\[ X = (1, a, x_1, \dots, x_n, y_1, \dots, y_k, z_1, \dots, z_l)^T \in \mathbb{R}^{N} \]
where the total number of variables is $N = 2 + n + k + l$. 

We construct the incidence matrix $M$ such that the relation $MX = 0$ captures all linear constraints of the extended grid. The matrix is composed of exactly $2n + 2$ rows and $N$ columns, partitioned into three distinct blocks:
\begin{itemize}
    \item \textbf{Block $L$ (2 rows)}: This block translates the global dimensions of the rectangle into two specific rows, $L_1$ and $L_2$:
    \begin{itemize}
        \item Row $L_1$ enforces $1 = \sum\limits_{i=1}^l z_i$. It contains a $-1$ at the index of the variable "$1$", a $1$ at the indices of each $z_i$, and $0$ elsewhere.
        \item Row $L_2$ enforces $a = \sum\limits_{i=1}^k y_i$. It contains a $-1$ at the index of $a$, a $1$ at the indices of each $y_i$, and $0$ elsewhere.
    \end{itemize}
    \item \textbf{Block $R_1$ ($n$ rows)}: For each $1 \leq i \leq n$, the $i$-th row enforces $x_i = \sum\limits_{y \in Y_i} y$. It has a $-1$ at the index of $x_i$ and a $1$ at the indices of all $y \in Y_i$.
    \item \textbf{Block $R_2$ ($n$ rows)}: For each $1 \leq i \leq n$, the $i$-th row enforces $x_i = \sum\limits_{z \in Z_i} z$, structured similarly to $R_1$.
\end{itemize}

By construction, the vector $X$ containing the geometric lengths of the tiling is a solution to this system, meaning $X \in \ker M$. Since its first coordinate is $1 \neq 0$, we have $\dim \ker M \geq 1$.

\begin{theorem} \label{thm:kernel_dim}
    For any square tiling, the associated matrix $M$ satisfies $\dim \ker M = 1$. Consequently, the incidence matrix uniquely determines the tiling up to overall scaling.
\end{theorem}

\begin{proof}
    Let $V = (v_1, v_a, v_{x_1}, \dots, v_{x_n}, v_{y_1}, \dots, v_{y_k}, v_{z_1}, \dots, v_{z_l})^T \in \ker M$.

    First, $k$ and $l$ are determined from $M$ thanks to the first two rows.\\
    The components of $V$ satisfy the additive relations of the grid. By replicating the algebraic manipulation from Theorem \ref{thm:rationality} (which relies solely on these additive relations), we obtain the pseudo-area conservation:
    \[ v_1 v_a = \sum_{j=1}^n v_{x_j}^2 \]
    
    Suppose that the first coordinate is zero, $v_1 = 0$. This implies $0 = \sum\limits_{j=1}^n v_{x_j}^2$. Since we are working over $\mathbb{R}$, this forces $v_{x_j} = 0$ for all $1 \leq j \leq n$.
    
    Now, assume there exists some horizontal segment variable such that $v_{y_i} \neq 0$. Let $i$ be the smallest such index. Geometrically, the segment $y_i$ terminates at a vertical line in the extended grid. By construction, there must be at least one square $x_j$ whose right edge lies on this vertical line.
    
    The horizontal span of this square, $Y_j$, is a contiguous subset of the horizontal segments ending exactly at $y_i$. Thus, $Y_j = \{y_m, y_{m+1}, \dots, y_i\}$ for some $m \leq i$.
    Since $V \in \ker M$, the $j$-th row in block $R_1$ enforces:
    \[ v_{x_j} = \sum_{h=m}^i v_{y_h} \]
    Because $i$ is the smallest index for which $v_{y_i} \neq 0$, all terms $v_{y_h}$ for $h < i$ are zero. This leaves $v_{x_j} = v_{y_i}$. However, since $v_{y_i} \neq 0$, this implies $v_{x_j} \neq 0$, which contradicts our earlier deduction that all $v_{x_j} = 0$.
    
    Consequently, no such $v_{y_i}$ can exist, and $v_{y_i} = 0$ for all $i$. By a similar argument, $v_{z_i} = 0$ for all $i$. Finally, the relation from $L_2$ gives $v_a = \sum v_{y_i} = 0$.
    
    Therefore, any vector in $\ker M$ with a first coordinate $v_1 = 0$ is necessarily the zero vector. If $\dim \ker M \geq 2$, we could form a non-trivial linear combination of two independent vectors to zero out the first coordinate, yielding a non-zero vector with $v_1 = 0$, which is impossible. 
    
    Thus, $\dim \ker M = 1$, and any solution is collinear to $X$.
\end{proof}
We can use Theorem \ref{thm:kernel_dim} to show Proposition \ref{prop:segments_bound} in a completely different, algebraic way:

\begin{proof}[Algebraic Proof of Proposition \ref{prop:segments_bound}]
    The incidence matrix $M$ has $N = n + k + l + 2$ columns and $2n + 2$ rows. 
    Since we have established that $\dim \ker M = 1$, by the Rank-Nullity theorem, the rank of $M$ is:
    \[ \operatorname{Rank}(M) = N - \dim \ker M = n + k + l + 1 \]
    A fundamental property of linear algebra is that the rank of a matrix cannot exceed its number of rows. Therefore, we have:
    \[ n + k + l + 1 \leq 2n + 2 \]
    Which yields:
    \[ k + l \leq n + 1 \]
    This concludes the algebraic proof of the geometric bound.
\end{proof}

With the matrix representation established, we can derive a combinatorial upper bound on the total number of valid square tilings of a given order $n$, and thus prove the finiteness of the number of tiling of order $n$.

\begin{theorem}\label{thm:combinatorial_bound}
    The number of distinct square tilings of order $n$ is finite, and upper bounded asymptotically by $O(c^n n^{3n+2})$ where $c = \frac{e}{64}$.
\end{theorem}

\begin{proof}
    By Theorem \ref{thm:kernel_dim}, a matrix uniquely determines the geometry of the tiling. Moreover, having fixed $k$ and $l$, the tiling is entirely encoded by the incidence blocks $R_1$ and $R_2$. Therefore, bounding the number of valid configurations for these blocks is sufficient to bound the total number of tilings.
    
    Consider an extended grid constructed from a tiling. It consists of $k$ horizontal variables $y_1, \dots, y_k$ and $l$ vertical variables $z_1, \dots, z_l$. Using Proposition \ref{prop:segments_bound}, we know that $k + l \leq n + 1$.
    
    In this grid, each square $x_i$ corresponds to a discrete rectangular region. The horizontal span of this region, denoted $Y_i$, must be a contiguous subsegment of $(y_1, \dots, y_k)$. The number of such contiguous subsegments of length at least $1$ is exactly $\binom{k+1}{2}$. Similarly, the vertical span $Z_i$ must be a contiguous subsegment of $(z_1, \dots, z_l)$, yielding $\binom{l+1}{2}$ possibilities.
    
    Therefore, the maximum number of distinct rectangular regions that can be formed in a $k \times l$ extended grid is:
    \[
    P_{k,l} = \binom{k+1}{2} \binom{l+1}{2} \leq \frac{(k+1)^2 (l+1)^2}{4}
    \]
    
    Subject to the constraint $k + l \leq n + 1$, the product $(k+1)(l+1)$ is maximized when $k = l = \frac{n+1}{2}$. This yields $P_{k,l} \leq \frac{(n+3)^4}{64}$.
    
    Using the inequality $\binom{N}{n} \leq \frac{N^n}{n!}$ and the lower bound from Stirling's formula $n! \geq (n/e)^n$, we can bound the binomial coefficient:
    \[ \binom{P_{k,l}}{n} \leq \frac{(P_{k,l})^n}{n!} \leq \frac{\left( \frac{(n+3)^4}{64} \right)^n}{(n/e)^n} = \left(\frac{e}{64}\right)^n \frac{(n+3)^{4n}}{n^n} \]
    
    We can rewrite the rightmost fraction by factoring out $n^{4n}$:
    \[ \frac{(n+3)^{4n}}{n^n} = n^{3n} \left(1 + \frac{3}{n}\right)^{4n} \]
    
    Using the classic analytical bound $(1 + x/n)^n \leq e^x$, we have $\left(1 + \frac{3}{n}\right)^{4n} \leq e^{12}$. Substituting this back, we obtain:
    \[ \binom{P_{k,l}}{n} \leq e^{12} \left(\frac{e}{64}\right)^n n^{3n} \]
    
    The total number of tilings is obtained by summing over all pairs $(k, l)$ where $k+l\leq n+1$. There is exactly $\frac{n(n+1)}{2}$ such pairs, which is upper bounded by $n^2$. Therefore, the total number of tilings satisfies:
    \[ N_{\text{total}} = \sum_{k,l} N_{\text{tilings,k,l}} \leq n^2 \cdot \left( e^{12} \left(\frac{e}{64}\right)^n n^{3n} \right) = e^{12} c^n n^{3n+2} \]
    where $c = \frac{e}{64}$ is a positive constant.
    
    Consequently, the upper bound is asymptotically $O(c^n n^{3n+2})$.
\end{proof}

\vspace{0.5cm}
\noindent \textbf{Remark (Planar Maps and better Asymptotics):} 
While our algebraic approach provides a simple bound of $O(c^nn^{3n+2})$ by enumerating all mathematically compatible matrix configurations, the true number of physical tilings is vastly smaller. Our matrix model counts many "intertwined" logical configurations that satisfy segment additions but cannot be drawn flat on a plane without overlapping. Brooks, Smith, Stone, and Tutte's approach \cite{BSST1940} yields a much better upper bound by counting planar graphs.

\section{Upper Bound on the Least Common Multiple}

Let $M$ be the incidence matrix defined previously, with dimensions $(2n+2) \times N$, where $N = n + k + l + 2$. We know $\dim \ker M = 1$. Let us extract a maximal rank submatrix $M'$ of size $(N-1) \times N$ by selecting $N-1$ linearly independent rows from $M$. 

By the cofactor method (or generalized cross product), the unique one-dimensional kernel of $M'$ is spanned by the integer vector $u \in \mathbb{Z}^N$ defined by:
$$ u_j = (-1)^j \det(M'_{j}) $$
where $M'_{j}$ is the square $(N-1) \times (N-1)$ matrix obtained by removing the $j$-th column of $M'$.

Since the vector $X \in \ker M = \ker M'$, we have $X = c \cdot u$ for some $c \in \mathbb{Q}$. If we multiply $X$ by $D = \operatorname{lcm}(\text{denominators}(a, x_1, \dots, x_n, y_1, \dots, y_k, z_1, \dots, z_l))$, we obtain a primitive integer vector. Since $u$ is also an integer vector in the same direction, the primitive vector must divide $u$. Thus, for any coordinate $j$, we have $D \cdot X_j \leq |u_j| = |\det(M'_{j})|$.\\
    Therefore, to bound the LCM of the denominators, it suffices to bound the determinant of $M'_{j}$.
\begin{theorem}\label{thm:den_bound}
    For any square tiling of order $n$, the least common multiple $D$ of the denominators of the $x_i,y_i,z_i$ and $a$ is bounded by $2^{k+l-1} \leq 2^n$.\\ 
    Furthermore, we obtain the alternative bound:
    $$ D \leq \frac{2^{k+l-2}}{x} $$
    where $x = \max \{ x_i \mid \text{both rows } R_{1,i} \text{ and } R_{2,i} \text{ are kept in the submatrix } M' \}$.
\end{theorem}

\begin{proof}
    The matrix $M$ has $2n+2$ rows, while the submatrix $M'$ has $N-1 = n+k+l+1$ rows. Therefore, exactly $(2n+2) - (n+k+l+1) = n + 1 - k - l$ rows from $M$ are removed to form $M'$.
    
    Because $X \in \ker M'$ consists entirely of positive lengths ($x_i, y_j, z_k, a > 0$), the cofactor vector $u$ cannot contain any zero entries. Consequently, $\det(M'_j) \neq 0$ for all $j$.
    
    Looking at the first coordinate (which is $1$ in $X$), we have $D \times 1 \leq |\det(M'_{1})|$. The matrix $M'_{1}$ is formed by removing the first column (the variable "$1$"). Its columns thus correspond to $(a, x_1, \dots, x_n)$ and $(y_1, \dots, y_k, z_1, \dots, z_l)$. 
    
    To prevent any determinant from evaluating to zero, $M'$ must not contain any column of zeros. Thus, to obtain $M'$ from $M$, the rows $L_1$ and $L_2$ can't be removed, as the column corresponding to "$1$" and $a$ in $M'$ would be entirely zero, making $\det(M'_j) = 0$ for all $j > 2$, which is impossible.
    Moreover, for each $1 \leq i \leq n$, the column for $x_i$ contains exactly two non-zero entries ($-1$s) located in rows $R_{1,i}$ and $R_{2,i}$. If both rows are removed, the column for $x_i$ would be zero. Thus, at most one row from each pair $(R_{1,i}, R_{2,i})$ can be removed.
    
    Since exactly $n + 1 - k - l$ rows are removed in total, and they must be selected from distinct row pairs $(R_{1,i}, R_{2,i})$, exactly $n - (n + 1 - k - l) = k + l - 1$ pairs remain fully intact in $M'$.
    
    When we expand $\det(M'_1)$ by multi-linearity on the $n$ columns corresponding to $x_1, \dots, x_n$, $n + 1 - k - l$ columns contain exactly one non-zero entry and $k + l - 1$ columns contain exactly two non-zero entries.

    This splits $\det(M'_1)$ into a sum of exactly $2^{k+l-1}$ determinants:
    $$ \det(M'_{1}) = \sum_{\text{choices}} \det(M'_{\text{choice}}) $$
    where each matrix $M'_{\text{choice}}$ is formed by choosing exactly one non-zero entry for each column $x_i$.
    
    In each of these $2^{k+l-1}$ matrices, the first $n+1$ columns contain exactly one $-1$ and zeros elsewhere. By developing successively along these first $n+1$ columns (standard Laplace expansion), the extracted signs yield $\det(M'_{\text{choice}}) = \pm 1 \times \det(S_{\text{choice}})$, where $S_{\text{choice}}$ is the remaining square submatrix formed by the complementary rows and the columns $(y_1, \dots, y_k, z_1, \dots, z_l)$.
    
    Crucially, because the geometric segments form contiguous boundaries of the rectangles, the $1$s appearing in any row of the $y$ and $z$ columns are consecutive. A matrix composed of $0$s and consecutive $1$s entries in each row is known as an \textit{interval matrix} and is \textit{totally unimodular} (see Ghouila-Houri \cite{GhouilaHouri1962} for example).
    
    Because $S_{\text{choice}}$ is a square submatrix of a totally unimodular matrix, its determinant must be in $\{-1, 0, 1\}$. Therefore:
    $$ |\det(M'_{\text{choice}})| \in \{0, 1\} $$
    
    By the triangle inequality applied to our initial sum of determinants, we conclude:
    $$ D \leq |\det(M'_{1})| \leq 2^{k+l-1} $$
    Since $k + l \leq n + 1$, it immediately follows that $D \leq 2^n$.

    Moreover, suppose there exists a square $x_i$ such that both rows $R_{1,i}$ and $R_{2,i}$ are kept in $M'$. Let us look at the coordinate $x_i$ of $X$. By the divisibility property of our primitive vector, we have:
    $$ D \cdot x_i \leq |u_{x_i}| = |\det(M'_{x_i})| $$
    
    The matrix $M'_{x_i}$ is formed by removing the column corresponding to $x_i$. The remaining $n-1$ columns of the $x$-type now contain exactly $(k + l - 1) - 1 = k + l - 2$ intact row pairs. 
    
    Following the exact same multi-linear expansion and total unimodularity argument as before, the determinant $\det(M'_{x_i})$ splits into exactly $2^{k+l-2}$ terms bounded by $1$. This yields:
    $$ D \cdot x_i \leq |\det(M'_{x_i})| \leq 2^{k+l-2} $$
    Thus, $D \leq \frac{2^{k+l-2}}{x_i}$. Since this inequality is true for all $x_i$ such that $R_{1,i}$ and $R_{2,i}$ are kept in the submatrix $M'$, it is also true for $x$.\\ 
    Thus, $D \leq \frac{2^{k+l-2}}{x}$. This bound is tighter than the general bound whenever $x > 1/2$.

    \textbf{Remark:} Looking at the coordinate $a$, we have $D \cdot a \leq |u_a| = |\det(M'_a)|$. The submatrix $M'_a$ is formed by removing the column for $a$. Since this column (like the column for "$1$") originally contained only a single non-zero entry, removing it leaves the exact same number of intact row pairs ($k+l-1$) for the variables $x_i$. Applying the same multi-linear expansion gives $|\det(M'_a)| \leq 2^{k+l-1}$, which yields $D \leq \frac{2^{k+l-1}}{a}$. While this inequality might appear to give a tighter bound when $a > 1$, it actually provides no new information as it is the same to applying our initial bound (from the first column) to the tiling after applying a homothety of ratio $1/a$, which scales the LCM $D$ proportionally.

\end{proof}

\section{Application: Recovering Kenyon's Bound}

To demonstrate the power of this algebraic framework, we can use our upper bound on the denominators to elegantly recover the main result established by Richard Kenyon in 1996 regarding the minimum number of squares required to tile a rectangle \cite{Kenyon1996}.

\begin{theorem}[Kenyon, 1996]
    Any tiling of a $p \times q$ rectangle (where $p, q \in \mathbb{N}^*$, $q > p$, and $\gcd(p,q) = 1$) by squares requires at least $\max(q/p, \log_2 q)$ squares.
\end{theorem}

\begin{proof}
    Let $n$ be the number of squares in such a tiling. We derive the two bounds independently.\\
    
    The total area of the rectangle is $p \times q$. Because the height of the rectangle is $p$, the side length of any square in the tiling cannot exceed $p$. Thus, the maximum area of a single square is $p^2$. The minimum number of squares required to cover the area $pq$ is therefore:
    $$ n \geq \frac{pq}{p^2} = \frac{q}{p} $$\\
    
    Let us apply a homothety to the rectangle by a factor of $1/q$. The dimensions of the rectangle become $(p/q) \times 1$. By orienting the rectangle such that its width is $1$ and its length is $a = p/q$, we fall exactly into the framework of our extended grid model.
    
    By Theorem \ref{thm:den_bound}, the least common multiple $D$ of the denominators of all geometric variables in this normalized tiling, which explicitly includes $a$, is bounded by $2^{k+l-1}\leq 2^n$ which implies:
    $$ q \leq 2^n $$
    
    Taking the base-2 logarithm of both sides yields:
    $$ n \geq \log_2 q $$
    
    Combining those two inequalities, we conclude that the number of squares must satisfy $n \geq \max(q/p, \log_2 q)$.
\end{proof}

\section{Generalization to Higher Dimensions}

The framework developed in this paper naturally extends to $d$-dimensional space. Let us consider a $d$-dimensional hypercuboid tiled by $n$ hypercubes of side lengths $x_1, \dots, x_n$. By applying a suitable homothety, we assume its dimensions are $a^{(1)} \times a^{(2)} \times \dots \times a^{(d-1)} \times 1$. 

The extended grid construction is generalize by extending the bounding hyperplanes of the $n$ hypercubes to the edges of the hypercuboid. Let $y^{(m)}_1, \dots, y^{(m)}_{k_m}$ denote the lengths of the segments along the $m$-th axis. 

The matrix $M$ is then constructed by defining $d+1$ distinct blocks of rows to encode the tiling:
\begin{itemize}
    \item \textbf{Global dimensions:} $1 = \sum_j y^{(d)}_j$ and $a^{(m)} = \sum_j y^{(m)}_j$ for all $1 \le m \le d-1$.
    \item \textbf{Incidence blocks ($d$ blocks of $n$ rows):} For each hypercube $i$ ($1 \le i \le n$) and each axis $m$ ($1 \le m \le d$), its side length $x_i$ is equal to the sum of its contiguous segments along that axis. This yields $d$ equations per hypercube: $x_i = \sum_{y \in Y_i^{(m)}} y^{(m)}$.
\end{itemize}

In this generalized matrix $M$, the column corresponding to each variable $x_i$ contains exactly $d$ non-zero entries (all $-1$), one for each dimensional block.

\begin{theorem} \label{thm:kernel_dim_d}
    For any $d$-dimensional hypercube tiling, the associated matrix $M$ satisfies $\dim \ker M = 1$ and uniquely determines the tiling up to overall scaling.
\end{theorem}

\begin{proof}
    As in the proof of Theorem \ref{thm:kernel_dim}, let $V \in \ker M$. Suppose its first coordinate, corresponding to the global scaled length "$1$", is equal to zero. We want to show that $V$ is the zero vector.
    
    Choose any axis $m \in \{1, \dots, d-1\}$ and pair it with the $d$-th axis. If we consider the intersection of the $d$-dimensional tiling with any plane parallel to the one spanned by the $m$-axis and the $d$-axis, we obtain a $2$-dimensional square tiling. The additive linear equations satisfied by the variables $x_i$, $y^{(m)}_j$, and $y^{(d)}_k$ in this plane are entirely contained within the rows of $M$. Since the global scaled length in $V$ is zero, applying the algebraic deduction from Theorem \ref{thm:kernel_dim} and because this plane is arbitrary guarantees that $v_{x_i} = 0$ for all $i$, and that all associated segment variables for these two axes are zero. Applying this pairwise planar argument to all axes ensures that all segment variables in $V$ are zero.
    
    Thus, $V = 0$, and we conclude that $\dim \ker M = 1$.
\end{proof}

\begin{theorem}\label{thm:den_bound_d}
    For any hypercube tiling of order $n$ in dimension $d$, the least common multiple $D$ of the denominators of the hypercubes' sizes is bounded by $d^n$.
\end{theorem}

\begin{proof}
    The proof is identical to the one of Theorem \ref{thm:den_bound}, relying on the total unimodularity of the remaining interval matrices \cite{GhouilaHouri1962}. Because each of the $n$ columns corresponding to the variables $x_i$ now contains exactly $d$ non-zero entries (one for each axis), the standard Laplace expansion splits the determinant into a sum of at most $d^n$ determinants which yields $D \leq d^n$.
\end{proof}

While the $d^n$ bound provides a self-contained and direct estimate, tighter asymptotic bounds have been recently established using more advanced techniques, such as planar multigraph complexity in two dimensions \cite{Fomin2021} and Diophantine approximations for higher dimensions \cite{Liu2023}.  

Our approach mainly offers an alternative, algebraic perspective. Because electrical network models rely on planar graphs, they do not easily generalize beyond two dimensions. In contrast, the properties of interval matrices \cite{GhouilaHouri1962} remain valid in $d$-dimensional space. We hope this straightforward linear encoding of geometric constraints will prove to be a convenient tool for studying higher-dimensional tilings.

\section*{Acknowledgements}

I would like to thank Romain Tessera for teaching the research course "Modal - Algèbre et géométrie - Pavages du plan" at École Polytechnique, during which this project was initiated. I would also like to thank Wassel Ben Yahia for introducing the topic and for our discussions exploring potential matrix representations of tilings.


\begin{thebibliography}{9}

\bibitem{Dehn1903} 
Max Dehn. 
\textit{Über die Zerlegung von Rechtecken in Rechtecke} (Sur la décomposition des rectangles en rectangles). 
Mathematische Annalen, vol. 57, pp. 314--332, 1903.

\bibitem{AignerZiegler}
Martin Aigner and Günter M. Ziegler.
\textit{Proofs from THE BOOK}.
Springer, Berlin, Heidelberg, 6th edition, 2018.

\bibitem{BSST1940} 
R. L. Brooks, C. A. B. Smith, A. H. Stone, and W. T. Tutte. 
\textit{The Dissection of Rectangles into Squares}. 
Duke Mathematical Journal, vol. 7, pp. 312--340, 1940.

\bibitem{GhouilaHouri1962}
Alain Ghouila-Houri.
\textit{Caractérisation des matrices totalement unimodulaires}.
Comptes Rendus Hebdomadaires des Séances de l'Académie des Sciences, Paris, vol. 254, pp. 1192--1194, 1962.

\bibitem{Kenyon1996}
Richard Kenyon.
\textit{Tiling a Rectangle with the Fewest Squares}.
Journal of Combinatorial Theory, Series A, vol. 76, no. 2, pp. 272--291, 1996.

\bibitem{Fomin2021}
D. Fomin.
\textit{Elementary proof for the bounds of the complexity of a planar multigraph and the size of a prime rectangular squaring}.
arXiv preprint arXiv:2103.10523, 2021.

\bibitem{Liu2023}
Tamás Keleti, Stephen Lacina, Changshuo Liu, Mengzhen Liu, and José Ramón Tuirán Rangel.
\textit{Tiling of rectangles with squares and related problems via Diophantine approximation}.
Discrete Mathematics, vol. 346, no. 9, article 113442, 2023. 
\url{https://doi.org/10.1016/j.disc.2023.113442}


\end{thebibliography}
\end{document}